\documentclass[reqno,12pt]{amsart}
\usepackage{bbm}
\usepackage{} 
\numberwithin{equation}{section}
\usepackage{extarrows}
\usepackage{amssymb}
\usepackage{mathrsfs}
\usepackage{xy}
\xyoption{all}
\def\Ext{\mbox{\rm Ext}\,} \def\Hom{\mbox{\rm Hom}} \def\dim{\mbox{\rm dim}\,} \def\Iso{\mbox{\rm Iso}\,}
\def\lr#1{\langle #1\rangle} \def\fin{\hfill$\square$}  \def\lra{\longrightarrow} 
     
  \def\Aut{\mbox{\rm Aut}\,}
\def\A{{\mathcal A}}\def\P{{\mathscr{P}}}
\theoremstyle{plain} 
\newtheorem{theorem}{\bf Theorem}[section]
\newtheorem{lemma}[theorem]{\bf Lemma}

\theoremstyle{definition} 
\newtheorem{definition}[theorem]{\bf Definition}
\newtheorem{remark}[theorem]{\bf Remark}

\begin{document}

\title[A note on Bridgeland's Hall algebras]{A note on Bridgeland's Hall algebras} 

\author[Haicheng Zhang]{{Haicheng Zhang}} 


%
\address{
Yau Mathematical Sciences Center, Tsinghua University,
 Beijing 100084, P. R. China\endgraf
}
\email{zhanghai14@mails.tsinghua.edu.cn}


\begin{abstract}
In this note, let $\A$ be a finitary hereditary abelian category with enough projectives. By using the associativity formula of Hall algebras, we give a new and simple proof of the main theorem in \cite{Yan}, which states that the Bridgeland's Hall algebra of 2-cyclic complexes of projective objects in $\A$ is isomorphic to the Drinfeld double Hall algebra of $\A$. In a similar way, we give a simplification of the key step in the proof of Theorem 4.11 in \cite{LP}.
\end{abstract}
\maketitle
\section{Introduction}
Ringel \cite{R90a} introduced the Hall algebra of a finite dimensional algebra over a finite field.
By the works of Ringel \cite{R90a,R92a,R95} and Green \cite{Gr95}, the twisted Hall algebra, called the Ringel--Hall algebra, of a finite dimensional hereditary algebra provides a realization of the positive (negative) part of the corresponding quantum group.
In order to obtain a Hall algebra description of the entire quantum group, one considers the Hall algebras of triangulated categories (for example, \cite{Kapranov}, \cite{Toen}, \cite{XiaoXu}). In \cite{Xiao}, Xiao gave a realization of the whole quantum group by constructing the Drinfeld double of the extended Ringel--Hall algebra of any hereditary algebra.

In 2013, Bridgeland \cite{Br} considered the Hall algebra of 2-cyclic complexes of projective modules over a finite dimensional hereditary algebra $A$, and achieved an algebra, called the (reduced) Bridgeland's Hall algebra of $A$, by taking some localization and reduction. He proved that there is an algebra embedding from the Ringel--Hall algebra of $A$ to its Bridgeland's Hall algebra. Moreover, the quantum group associated with \emph{A} can be embedded into the reduced Bridgeland's Hall algebra of $A$. This provides a realization of the full quantum group by Hall algebras. In \cite{Br}, Bridgeland stated without proof that the Bridgeland's Hall algebra of each finite dimensional hereditary algebra is isomorphic to the Drinfeld double of its extended Ringel--Hall algebra. Later on, Yanagida proved this statement in \cite{Yan}. With the purpose of generalizing Bridgeland's construction to a bigger class of exact categories, Gorsky \cite{Gor2} defined the so-called semi-derived Hall algebra of the category of bounded complexes of $\mathcal {E}$ for each exact category $\mathcal {E}$ satisfying certain finiteness conditions. In particular, if every object in $\mathcal {E}$ has finite projective resolution, he gave a similar construction to the category of 2-cyclic complexes of $\mathcal {E}$.
Recently, inspired by the works of Bridgeland and Gorsky, Lu and Peng \cite{LP} have generalized Bridgeland's construction to any hereditary abelian category $\A$ which may not have enough projectives, and defined an algebra for the category of 2-cyclic complexes of $\A$, called the modified Ringel--Hall algebra of $\A$. They also proved that the resulting algebra is isomorphic to the Drinfeld double Hall algebra of $\A$.

The key step in
the proof that the Bridgeland's Hall algebra or modified Ringel--Hall algebra of a hereditary abelian category $\A$ is isomorphic to its Drinfeld double Hall algebra is to check the (Drinfeld) commutator relations.
The method used in \cite{Yan} is to make the summations on the left-hand and right-hand sides of the commutator relations symmetric by some analysis of the structure of the category of 2-cyclic complexes of projectives in $\A$, as well as some complicated calculations. It seems that the process of the proof is not too intuitive. While Lu and Peng gave a characterization of some coefficients in the commutator relations by introducing two sets, and obtained the proof by means of the coincidence of the cardinalities of these two sets. Nevertheless, their characterization is a bit complicated. In this note we use the associativity formula of Hall algebras to give a more intuitive and simpler proof that the Bridgeland's Hall algebra of $\A$ is isomorphic to its Drinfeld double Hall algebra. Similarly, based on \cite{LP}, we give a simplification of the key step in the proof of Theorem 4.11 in \cite{LP}.
Explicitly, we prove the commutator relations therein by using the associativity formula of Hall algebras rather than Lemma 4.10 in \cite{LP}.

Let us fix some notations used throughout the paper.
$k$ is always a finite field with $q$ elements and set $v=\sqrt{q}$. $\A$ is always a finitary hereditary abelian $k$-category with enough projectives unless otherwise stated, we also assume that the image $\hat{M}$ of $M$ in the Grothendieck group $K_0(\mathcal{A})$ is nonzero for any nonzero object $M$ in $\mathcal{A}$ (cf. \cite{Br,Yan}).
We denote by $\Iso(\mathcal{A})$ the set of isoclasses (isomorphism classes) of objects in $\mathcal{A}$. The subcategory of $\mathcal{A}$ consisting of projective objects is denoted by $\mathscr{P}$.
For a complex $M_\bullet=(\cdots\rightarrow M_{i+1}\xrightarrow{d_{i+1}} M_{i}\rightarrow\cdots)$ in $\mathcal{A}$, its homology is denoted by $H_\ast(M_\bullet)$.
For a finite set $S$, we denote by $|S|$ its cardinality. For an object $M\in\A$, we denote by $\Aut_{\A}(M)$ the automorphism group of $M$, and set $a_M=|\Aut_{\A}(M)|$.


\section{Preliminaries}

In this section, we collect some necessary definitions and properties. All of the materials can be
found in \cite{Br}, \cite{Sc} and \cite{Yan}.

\subsection{2-cyclic complexes}
Let $\mathcal{C}_2(\mathcal{A})$ be the abelian category of 2-cyclic complexes over $\mathcal{A}$. The objects of this category consist of diagrams
\begin{equation*}M_\bullet=\xymatrix{{M_1}\ar@<0.7ex>[r]^{d_1^M}&{M_0}\ar@<0.7ex>[l]^{d_0^M}}\end{equation*}
in $\mathcal{A}$ such that $d_{i+1}^M\circ d_{i}^M=0, ~i\in \mathbb{Z}_2$.
A morphism $s_\bullet: M_\bullet \rightarrow N_\bullet$ consists of a diagram
\begin{equation*}\xymatrix{{M_1}\ar@<0.7ex>[r]^{d_1^M}\ar[d]_{s_1}&{M_0}\ar@<0.7ex>[l]^{d_0^M}\ar[d]^{s_0}\\
{N_1}\ar@<0.7ex>[r]^{d_1^N}&{N_0}\ar@<0.7ex>[l]^{d_0^N}}\end{equation*}
with $s_{i+1}\circ d_i^M=d_i^N\circ s_i, ~i\in \mathbb{Z}_2$.
Two morphisms $s_\bullet, t_\bullet : M_\bullet \rightarrow N_\bullet$ are said to be \emph{homotopic} if there are morphisms $h_i : M_i \rightarrow N_{i+1},~i\in \mathbb{Z}_2,$ such that $t_i-s_i={d_{i+1}^N}\circ h_i+{h_{i+1}}\circ d_i^M,~i\in \mathbb{Z}_2$.
For an object $M_\bullet\in\mathcal{C}_2(\mathcal{A})$, we define its class in the Grothendieck group $K_0(\mathcal{A})$ to be
$$\hat{M_\bullet} := \hat{M_0}-\hat{M_1}\in K_0(\mathcal{A}).$$
Denote by ${\mathcal {K}}_2(\mathcal{A})$ the homotopy category obtained from $\mathcal{C}_2(\mathcal{A})$ by identifying homotopic morphisms. Denote by $\mathcal{C}_2(\mathscr{P}) \subset \mathcal{C}_2(\mathcal{A})$ the full subcategory whose objects are complexes of projectives in $\mathcal{A}$, and by ${\mathcal {K}}_2(\mathscr{P})$ its homotopy category.
The shift functor of complexes induces an involution of $\mathcal{C}_2(\mathcal{A})$. This involution
shifts the grading and changes the signs of differentials as follows
\[{M_ \bullet } = \xymatrix{{M_1}\ar@<0.7ex>[r]^{d_1^M}&{M_0}\ar@<0.7ex>[l]^{d_0^M}}\overset {*} \longleftrightarrow M_ \bullet ^* = \xymatrix{{M_0}\ar@<0.7ex>[r]^{-d_0^M}&{M_1}\ar@<0.7ex>[l]^{-d_1^M}.}\]

Let $\mathcal{D}^b(\mathcal{A})$ be the bounded derived category of $\mathcal{A}$, with the suspension functor $[1]$. Let $\mathcal{R}_2(\mathcal{A})$ = $\mathcal{D}^b(\mathcal{A})/[2]$ be the orbit category, also known as the root category of $\mathcal{A}$. The category $\mathcal{D}^b(\mathcal{A})$ is equivalent to the bounded homotopy category $K^b(\mathscr{P})$, since $\mathcal{A}$ is hereditary. In this case, we can equally well define  $\mathcal{R}_2(\mathcal{A})$ as the orbit category of $K^b(\mathscr{P})$.

\begin{lemma}{\rm(\cite{PX97}, \cite[Lemma 3.1]{Br})}\label{fully faithful}
There is an equivalence $D : \mathcal{R}_2(\mathcal{A})\rightarrow {\mathcal {K}}_2(\mathscr{P})$ sending a bounded complex of projectives $(P_i)_{i \in \mathbb{Z}}$ to the $2$-cyclic complex
$$\xymatrix{{\bigoplus_{i\in\mathbb{Z}} P_{2i+1}}\ar@<0.7ex>[r]^{}&{\bigoplus_{i\in\mathbb{Z}} P_{2i}}\ar@<0.7ex>[l]^{}.}$$
\end{lemma}

\begin{lemma}{\rm(\cite[Lemma 3.3]{Br})}\label{Ext to Hom}
If $M_\bullet,N_\bullet \in \mathcal{C}_2(\mathscr{P})$, then there exists an isomorphism of vector spaces $$\Ext_{\mathcal{C}_2(\mathcal{A})}^1(N_\bullet,M_\bullet) \cong \Hom_{{\mathcal {K}}_2(\mathcal{A})}(N_\bullet,M_\bullet^\ast).$$
\end{lemma}

A complex $M_\bullet \in \mathcal{C}_2(\mathcal{A})$ is called \emph{acyclic} if $H_\ast(M_\bullet)=0$. Each object $P \in \mathscr{P}$ determines acyclic complexes
\[{K_P} = (\xymatrix{{P}\ar@<0.7ex>[r]^{1}&{P}\ar@<0.7ex>[l]^{0}}),  ~~~~~~~K_P^* = (\xymatrix{{P}\ar@<0.7ex>[r]^{0}&{P}\ar@<0.7ex>[l]^{1}}).\]

\begin{lemma}{\rm(\cite[Lemma 3.2]{Br})}\label{zero}
For each acyclic complex $M_\bullet \in \mathcal{C}_2(\mathscr{P})$, there are objects $P,Q \in \mathscr{P}$, unique up to isomorphism, such that $M_\bullet \cong K_P \bigoplus K_Q^*$.
\end{lemma}

\subsection{Hall algebras}
 Given objects $L,M,N \in \mathcal{A}$, let $\Ext_\mathcal{A}^1(M,N)_L \subset \Ext_\mathcal{A}^1(M,N)$ be the subset consisting of those equivalence classes of short exact sequences with middle term $L$.
\begin{definition}\label{Hall algebra of abelian category}
The \emph{Hall algebra} $\mathcal {H}(\mathcal{A})$ of $\mathcal{A}$ is the vector space over $\mathbb{C}$ with basis elements $[M] \in \Iso(\mathcal{A}$), and with multiplication defined by
\[[M] \diamond [N] = \sum\limits_{[L] \in \Iso(\mathcal{A})} {\frac{{|\Ext_\mathcal{A}^1{{(M,N)}_L}|}}{{|\Hom_\mathcal{A}(M,N)|}}} [L].\]
\end{definition}
By \cite{R90a}, the above operation $\diamond$ defines on $\mathcal {H}(\mathcal{A})$ the structure of a unital associative algebra over $\mathbb{C}$, and the class [0] of the zero object is the unit.
\begin{remark}
Given objects $L,M,N\in \A$, set
$$g_{MN}^L=|\{N'\subset L~|~N'\cong N, L/N'\cong M\}|.$$
By Riedtmann-Peng formula \cite{Riedtmann,Peng},
$$g_{MN}^{L}=\frac{|\Ext^1_{\A}(M,N)_{L}|}{|\Hom_{\A}(M,N)|}\cdot \frac{a_{L}}{a_{M}a_{N}}.$$ Thus in terms of alternative generators $[[M]]=\frac{[M]}{a_M}$, the product takes the form
$$[[M]]\diamond [[N]]= \sum\limits_{[L] \in \Iso(\mathcal{A})}g_{MN}^L[[L]],$$
which is the definition used, for example, in \cite{R90a,Sc}. The associativity of Hall algebras amounts to the following formula
\begin{equation}\label{jiehe}\sum\limits_{[M] \in \Iso(\mathcal{A})}g_{XY}^Mg_{MZ}^L=\sum\limits_{[N] \in \Iso(\mathcal{A})}g_{XN}^Lg_{YZ}^N(=:g_{XYZ}^L),\end{equation} for any objects $L,X,Y,Z\in\A$.
\end{remark}


For objects $M,N \in \mathcal{A}$, let $$\lr{M,N}:=\dim_k\Hom_{\A}(M,N)-\dim_k\Ext^1_{\A}(M,N),$$
and it descends to give a bilinear form
$$\lr{\cdot ,\cdot }: K_0(\mathcal{A})\times K_0(\mathcal{A})\longrightarrow \mathbb{Z},$$ known as the \emph{Euler form}. We also consider the \emph{symmetric Euler form}
$$(\cdot ,\cdot ): K_0(\mathcal{A})\times K_0(\mathcal{A})\longrightarrow \mathbb{Z},$$ defined by $(\alpha,\beta)=\lr{\alpha,\beta}+\lr{\beta,\alpha}$ for all $\alpha,\beta \in K_0(\mathcal{A})$.
The \emph{Ringel--Hall algebra} ${\mathcal {H}}_{\rm{tw}}(\mathcal{A})$ of $\mathcal{A}$ is the same vector space as $\mathcal {H}(\mathcal{A})$, but with multiplication defined by $$[M]\ast[N]=v^{\lr{\hat{M},\hat{N}}}\cdot[M]\diamond[N].$$
The \emph{extended Ringel--Hall algebra} $\tilde{\mathcal {H}}(\A)$ of $\A$ is defined as an extension of ${\mathcal {H}}_{\rm{tw}}(\mathcal{A})$ by adjoining symbols $K_{\alpha}$ for $\alpha\in K(\A)$, and imposing relations $$K_{\alpha}K_{\beta}=K_{\alpha+\beta},\quad K_{\alpha}[M]=v^{\lr{\alpha,\hat{M}}}[M]K_{\alpha},$$ for $\alpha,\beta\in K_0(\A)$ and $[M]\in \Iso(\mathcal{A})$.

By Green \cite{Gr95} and Xiao \cite{Xiao}, the extended Ringel--Hall algebra $\tilde{\mathcal {H}}(\A)$
is a topological bialgebra (see \cite{Sc}) with comultiplication $\Delta$ and counit $\epsilon$ defined by
$$\Delta([L]K_{\alpha})=\sum\limits_{[M],[N] \in \Iso(\mathcal{A})}v^{\lr{\hat{M},\hat{N}}}g_{MN}^L[M]K_{\alpha+\hat{N}}\otimes [N]K_{\alpha}\quad\mbox{and}\quad \epsilon([L]K_{\alpha})=\delta_{L,0}.$$
It is well known that there exists a nondegenerate symmetric bilinear
$$\varphi(-,-): \tilde{\mathcal {H}}(\A)\times\tilde{\mathcal {H}}(\A)\longrightarrow
\mathbb{C},$$ defined by
$$\varphi([M]K_{\alpha},[N]K_{\beta})=\delta_{[M],[N]}a_{M}v^{(\alpha,\beta)}.$$ This is a Hopf pairing (see for example \cite{Gr95,Sc,Xiao}).
Then the {\em Drinfeld double Hall algebra} $D(\A)$ of $\A$ is by definition the free product $\tilde{\mathcal {H}}(\A)\ast\tilde{\mathcal {H}}(\A)$ divided out by the commutator relations
(with $a,b\in\tilde{\mathcal {H}}(\A)$) \begin{equation}\label{Drinfeld}\sum\varphi(a_{(2)},b_{(1)})\cdot a_{(1)}\otimes b_{(2)}=\sum\varphi(a_{(1)},b_{(2)})(1\otimes b_{(1)})\circ (a_{(2)}\otimes1).\end{equation} Here we use Sweedler's notation: $\Delta(a)=\sum a_{(1)}\otimes a_{(2)}$.

%
%
%

\subsection{Bridgeland's Hall algebras}
Let $\mathcal {H}(\mathcal{C}_2(\mathcal{A}))$ be the Hall algebra of the abelian category $\mathcal{C}_2(\mathcal{A})$ defined in Definition \ref{Hall algebra of abelian category} and $\mathcal {H}(\mathcal{C}_2(\mathscr{P})) \subset \mathcal {H}(\mathcal{C}_2(\mathcal{A}))$ be the subspace spanned by the isoclasses of complexes of projective objects. Define $\mathcal {H}_{\rm{tw}}(\mathcal{C}_2(\mathscr{P}))$ to be the same vector space as $\mathcal {H}(\mathcal{C}_2(\mathscr{P}))$ with ``twisted" multiplication defined by $$[M_\bullet]\ast[N_\bullet]:=v^{\lr{\hat{M}_0,\hat{N}_0}+\lr{\hat{M}_1,\hat{N}_1}}\cdot[M_\bullet]\diamond[N_\bullet].$$ Then $\mathcal {H}_{\rm{tw}}(\mathcal{C}_2(\mathscr{P}))$ is an associative algebra (see \cite{Br}).

We have the following simple relations for the acyclic complexes ${K_P}$ and ${K_P^*}$.

\begin{lemma}{\rm(\cite[Lemma 3.5]{Br})}\label{formula}
For any object $P \in \mathscr{P}$ and any complex ${M_\bullet}\in\mathcal{C}_2(\mathscr{P})$, we have the following relations in $\mathcal {H}_{\rm{tw}}(\mathcal{C}_2(\mathscr{P}))$
\begin{alignat}{2}
&[K_P]\ast[M_\bullet]=v^{\lr{\hat{P},\hat{M}_\bullet}}[K_P \oplus M_\bullet],&\quad&[M_\bullet]\ast[K_P]=v^{-\lr{\hat{M}_\bullet,\hat{P}}}[K_P \oplus M_\bullet];\\
&[K_P^\ast]\ast[M_\bullet]=v^{-\lr{\hat{P},\hat{M}_\bullet}}[K_P^\ast \oplus M_\bullet],&\quad&[M_\bullet]\ast[K_P^\ast]=v^{\lr{\hat{M}_\bullet,\hat{P}}}[K_P^\ast \oplus M_\bullet];\\
&[K_P]\ast[M_\bullet]=v^{(\hat{P},\hat{M}_\bullet)}[M_\bullet]\ast[K_P],&\quad&[K_P^\ast]\ast[M_\bullet]=v^{-(\hat{P},\hat{M}_\bullet)}[M_\bullet]\ast[K_P^\ast]\label{community of K_P with others}.
\end{alignat}
In particular, for $P,Q \in \mathscr{P}$, we have
\begin{flalign}&[K_P]\ast[K_Q]=[K_P \oplus K_Q],~~~~[K_P]\ast[K_Q^\ast]=[K_P \oplus K_Q^\ast];\\
&[[K_P],[K_Q]]=[[K_P],[K_Q^\ast]]=[[K_P^\ast],[K_Q^\ast]]=0.\end{flalign}
\end{lemma}

By Lemmas \ref{zero} and \ref{formula}, the acyclic elements of $\mathcal {H}_{\rm{tw}}(\mathcal{C}_2(\mathscr{P}))$ satisfy the Ore conditions and thus we  have the following definition from \cite{Br}.
\begin{definition}
The \emph{Bridgeland's Hall algebra} of $\mathcal{A}$, denoted by $\mathcal {D}\mathcal {H}(\mathcal{A})$, is the localization of $\mathcal {H}_{\rm{tw}}(\mathcal{C}_2(\mathscr{P}))$ with respect to the elements $[M_\bullet]$ corresponding to acyclic complexes $M_\bullet$. In symbols, $$\mathcal {D}\mathcal {H}(\mathcal{A}):=\mathcal {H}_{\rm{tw}}(\mathcal{C}_2(\mathscr{P}))[~[M_\bullet]^{-1}~|~H_\ast(M_\bullet)=0~].$$
\end{definition}
As explained in \cite{Br}, this is the same as localizing by the elements $[K_P]$ and $[K_P^\ast]$ for all objects $P \in \mathscr{P}.$
Writing $\alpha \in K_0(\mathcal{A})$ in the form $\alpha = \hat{P}-\hat{Q}$ for some objects $P,Q \in \mathscr{P}$, one defines $K_\alpha = [K_P]\ast[K_Q]^{-1}, K_\alpha^\ast = [K_P^\ast]\ast[K_Q^\ast]^{-1}$. Note that the equalities in (\ref{community of K_P with others}) continue to hold with the elements $[K_P]$ and $[K_P^\ast]$ replaced by $K_\alpha$ and $K_\alpha^\ast$, respectively, for any $\alpha \in K_0(\mathcal{A})$.

For each object $M \in \mathcal{A}$, by \cite[Lemma 4.1]{Br}, we fix a minimal projective resolution\footnote{The notations $P_M$ and $\Omega_M$ will be used throughout the paper.} of the form
\begin{equation}\label{projective resolution}
0 \longrightarrow \Omega_M \stackrel{\delta_M}{\longrightarrow} P_M \longrightarrow M \longrightarrow 0.
\end{equation}
Set \begin{equation*}\label{C_A}
C_M:=\xymatrix{{\Omega_M}\ar@<0.7ex>[r]^-{
  \delta_M
}&**[r]{P_M.}\ar@<0.7ex>[l]^-{
  0
}}
\end{equation*}

Since the minimal projective resolution of $M$ is unique up to isomorphism, the complex $C_M$ is well-defined up to isomorphism.
\begin{lemma}{\rm(\cite[Lemma 4.2]{Br})}
Each object ${L_\bullet}\in\mathcal{C}_2(\mathscr{P})$ has a direct sum decomposition
$$L_\bullet=C_M\oplus C_N^\ast\oplus K_P\oplus K_Q^\ast.$$
Moreover, the objects $M,N\in\A$ and $P,Q\in\P$ are uniquely determined up to isomorphism.
\end{lemma}

As in \cite{Br}, we have an element $E_M$ in $\mathcal {D}\mathcal {H}(\mathcal{A})$ defined by
$$E_M :=v^{\lr{\hat{\Omega}_M,\hat{M}}}\cdot K_{-\hat{\Omega}_M} \ast [C_M] \in \mathcal {D}\mathcal {H}(\mathcal{A}).$$
It is easy to see that the shift functor defines an algebra involution $\ast$ on $\mathcal {D}\mathcal {H}(\mathcal{A})$. Set $F_M = E_M^\ast$ for any object $M \in \mathcal{A}$.


\begin{theorem}{\rm(\cite[Lemmas 4.6,4.7]{Br})}\label{embedding0}
The maps
\begin{equation*}
\begin{split}&I_{+}^e:\tilde{\mathcal {H}}(\A)\hookrightarrow \mathcal {D}\mathcal {H}(\mathcal{A}), ~~~~~[M]\mapsto E_M,~~~~K_{\alpha}\mapsto K_{\alpha};\\
&I_{-}^e:\tilde{\mathcal {H}}(\A)\hookrightarrow \mathcal {D}\mathcal {H}(\mathcal{A}), ~~~~~[M]\mapsto F_M,~~~~~K_{\alpha}\mapsto K_{\alpha}^\ast
\end{split}
\end{equation*}
are both embeddings of algebras. Moreover, the multiplication map
$m:a\otimes b\mapsto I_+^e(a)\ast I_-^e(b)$ defines an isomorphism of vector spaces
$$\xymatrix{m:\tilde{\mathcal {H}}(\A)\otimes\tilde{\mathcal {H}}(\A)\ar[r]^-{\simeq}& \mathcal {D}\mathcal {H}(\mathcal{A}).}$$
\end{theorem}

\section{Main Theorem}
In this section, we first present the main theorem which was stated by Bridgeland in \cite{Br}, proved by Yanagida in \cite{Yan}, and generalized by Lu and Peng in \cite{LP}. Then we provide a new and succinct proof by using the associativity formula of Hall algebras.

\noindent$\mathbf{Main~~Theorem}$~~(\cite{Br},\cite{Yan},\cite{LP})
\emph{The Bridgeland's Hall algebra $\mathcal {D}\mathcal {H}(\mathcal{A})$ is isomorphic to the Drinfeld double Hall algebra $D(\A)$.}

In what follows, we will give the proof of Main Theorem.
By Theorem \ref{embedding0}, it suffices to prove that the commutator relation
\begin{equation}\label{Drinfeld2}\sum\varphi(a_{(2)},b_{(1)})I_+^e(a_{(1)})\ast I_-^e(b_{(2)})=\sum\varphi(a_{(1)},b_{(2)})I_-^e(b_{(1)})\ast I_+^e(a_{(2)})\end{equation} holds in $\mathcal {D}\mathcal {H}(\mathcal{A})$ for each $a=[A]K_{\alpha}$ and $b=[B]K_{\beta}$ with $\alpha,\beta\in K_0(\A)$ and $[A],[B]\in\Iso(\A)$. By writing out the comultiplications $\Delta([A]K_{\alpha})$ and $\Delta([B]K_{\beta})$, and substituting into $(\ref{Drinfeld2})$, we find that we only need to prove that Relation $(\ref{Drinfeld2})$ holds for $a=[A]$ and $b=[B]$.

Since $$\Delta([A])=\sum\limits_{[A_1],[A_2]}v^{\lr{A_1,A_2}}g_{A_1A_2}^A[A_1]K_{\hat{A}_2}\otimes [A_2];$$
$$\Delta([B])=\sum\limits_{[B_1],[B_2]}v^{\lr{B_1,B_2}}g_{B_1B_2}^B[B_1]K_{\hat{B}_2}\otimes [B_2],$$
the left hand side of $(\ref{Drinfeld2})$ becomes
\begin{equation*}\begin{split}
\mbox{LHS~~of}~~(\ref{Drinfeld2})&=\sum\limits_{[A_1],[A_2],[B_1],[B_2]}v^{\lr{A_1,A_2}+\lr{B_1,B_2}}g_{A_1A_2}^A
g_{B_1B_2}^B\varphi([A_2],[B_1]K_{\hat{B}_2})E_{A_1}K_{\hat{A}_2}F_{B_2}\\
&=\sum\limits_{[A_1],[A_2],[B_2]}v^{\lr{A_1,A_2}+\lr{A_2,B_2}}g_{A_1A_2}^A
g_{A_2B_2}^Ba_{A_2}E_{A_1}K_{\hat{A}_2}F_{B_2}
\end{split}\end{equation*}
\begin{equation*}\begin{split}
&=\sum\limits_{[A_1],[A_2],[B_2]}v^{\lr{A_1,A_2}+\lr{A_2,B_2}}g_{A_1A_2}^A
g_{A_2B_2}^Ba_{A_2}v^{\lr{\Omega_{A_1},A_1}}K_{-\hat{\Omega}_{A_1}}[C_{A_1}]K_{\hat{A}_2}v^{\lr{\Omega_{B_2},B_2}}
K_{-\hat{\Omega}_{B_2}}^\ast[C_{B_2}^\ast]\\
&=\sum\limits_{[A_1],[A_2],[B_2]}v^{x}g_{A_1A_2}^A
g_{A_2B_2}^Ba_{A_2}K_{\hat{A}_2-\hat{\Omega}_{A_1}}K_{-\hat{\Omega}_{B_2}}^\ast[C_{A_1}][C_{B_2}^\ast],
\end{split}\end{equation*}
where \begin{equation*}\begin{split}x=&\lr{A_1,A_2}+\lr{A_2,B_2}+\lr{\Omega_{A_1},A_1}+\lr{\Omega_{B_2},B_2}-(A_2,A_1)-(\Omega_{B_2},A_1)\\
&=\lr{\hat{A}_2+\hat{\Omega}_{B_2},\hat{B}_2-\hat{A}_{1}}+\lr{\hat{\Omega}_{A_1},\hat{A}_1}-
\lr{\hat{A}_1,\hat{\Omega}_{B_2}}.
\end{split}\end{equation*}

\begin{equation*}\begin{split}
\mbox{RHS~~of}~~(\ref{Drinfeld2})&=\sum\limits_{[A_1],[A_2],[B_1],[B_2]}v^{\lr{A_1,A_2}+\lr{B_1,B_2}}g_{A_1A_2}^A
g_{B_1B_2}^B\varphi([A_1]K_{\hat{A}_2},[B_2])F_{B_1}K_{\hat{B}_2}^*E_{A_2}\\
&=\sum\limits_{[A'_1],[A'_2],[B'_1]}v^{\lr{A'_1,A'_2}+\lr{B'_1,A'_1}}g_{A'_1A'_2}^A
g_{B'_1A'_1}^Ba_{A'_1}F_{B'_1}K_{\hat{A}'_1}^*E_{A'_2}.
\end{split}\end{equation*}
\begin{equation*}\begin{split}
&=\sum\limits_{[A'_1],[A'_2],[B'_1]}v^{\lr{A'_1,A'_2}+\lr{B'_1,A'_1}}g_{A'_1A'_2}^A
g_{B'_1A'_1}^Ba_{A'_1}v^{\lr{\Omega_{B'_1},B'_1}}K_{-\hat{\Omega}_{B'_1}}^\ast[C_{B'_1}^\ast]K_{\hat{A}'_1}^*
v^{\lr{\Omega_{A'_2},A'_2}}K_{-\hat{\Omega}_{A'_2}}[C_{A'_2}]\\
&=\sum\limits_{[A'_1],[A'_2],[B'_1]}v^{x'}g_{A'_1A'_2}^A
g_{B'_1A'_1}^Ba_{A'_1}K_{\hat{A}'_1-\hat{\Omega}_{B'_1}}^\ast K_{-\hat{\Omega}_{A'_2}}[C_{B'_1}^\ast]
[C_{A'_2}],
\end{split}\end{equation*}
where \begin{equation*}\begin{split}
x'=&\lr{A'_1,A'_2}+\lr{B'_1,A'_1}+\lr{\Omega_{B'_1},B'_1}+\lr{\Omega_{A'_2},A'_2}-(A'_1,B'_1)-(\Omega_{A'_2},B'_1)\\
&=\lr{\hat{A}'_1+\hat{\Omega}_{A'_2},\hat{A}'_2-\hat{B}'_1}+\lr{\hat{\Omega}_{B'_1},\hat{B}'_1}-\lr{\hat{B}'_1,\hat{\Omega}_{A'_2}}.
\end{split}\end{equation*}

\begin{lemma}\label{zhongjian}
For any objects $X,Y\in\A$ and $T,W\in\P$. In $\mathcal {D}\mathcal {H}(\mathcal{A})$ we have
$$[C_X\oplus C_Y^\ast\oplus K_T\oplus K_W^\ast]=
v^{\lr{\hat{W}-\hat{T},\hat{X}-\hat{Y}}}K_{\hat{T}}K_{\hat{W}}^\ast[C_X\oplus C_Y^\ast].$$
\end{lemma}
\begin{proof}
By the commutative diagram
$$\xymatrix@R=0.6in @C=0.5in {{T\oplus W}\ar@<0.7ex>[r]^{\left({\begin{smallmatrix}{1}&\\&0\end{smallmatrix}}\right)}
\ar@<2.2ex>[d]_{\left({\begin{smallmatrix}{a}&b\\{c}&{d}\end{smallmatrix}}\right)}
&{T\oplus W}\ar@<0.7ex>[l]^{\left({\begin{smallmatrix}{0}&\\&{1}\end{smallmatrix}}\right)}
\ar@<0.7ex>[d]^{\left({\begin{smallmatrix}{a'}&b'\\{c'}&d'\end{smallmatrix}}\right)}\\
{\Omega_X\oplus P_Y}\ar@<0.7ex>[r]^{\left({\begin{smallmatrix}{\delta_X}&\\&0\end{smallmatrix}}\right)}&{P_X\oplus \Omega_Y,}\ar@<0.7ex>[l]^{\left({\begin{smallmatrix}{0}&\\&-\delta_Y\end{smallmatrix}}\right)}}$$
we easily obtain that $$|\Hom_{\mathcal{C}_2(\mathcal{A})}(K_T\oplus K_W^\ast,C_X\oplus C_Y^\ast)|=
q^{\lr{\hat{T},\hat{\Omega}_X+\hat{P}_Y}+\lr{\hat{W},\hat{P}_X+\hat{\Omega}_Y}}.$$
Hence, $$[K_T\oplus K_W^\ast][C_X\oplus C_Y^\ast]=v^m[C_X\oplus C_Y^\ast\oplus K_T\oplus K_W^\ast],$$
where \begin{equation*}\begin{split}m&=\lr{\hat{T}+\hat{W},\hat{P}_X+\hat{\Omega}_X+\hat{P}_Y+\hat{\Omega}_Y}-2\lr{\hat{T},\hat{\Omega}_X+\hat{P}_Y}-2\lr{\hat{W},\hat{P}_X+\hat{\Omega}_Y}\\
&=\lr{\hat{T}-\hat{W},\hat{X}-\hat{Y}}.
\end{split}\end{equation*}
\end{proof}

For any fixed objects $A,B,X,Y\in\A$, we denote by $W_{AB}^{XY}$ the set
$$\{(f,g,h)~|~\xymatrix{0\ar[r]&X\ar[r]^f&A\ar[r]^g&B\ar[r]^h&Y\ar[r]&0}~~\mbox{is~~exact~~in}~~\A\},$$
and denote by ${}_X\Hom_{\A}(A_1,B_2)_Y$ the set $$\{g~|~\xymatrix{0\ar[r]&X\ar[r]&A\ar[r]^g&B\ar[r]&Y\ar[r]&0}~~\mbox{is exact in}~~\A\}.$$
By \cite[(8.8)]{Vanden},  $$|W_{AB}^{XY}|=a_Xa_Y\sum\limits_{[L]}a_Lg_{LX}^Ag_{YL}^B,$$
and it is easy to see that \begin{equation}\label{lHoml}|{}_X\Hom_{\A}(A,B)_Y|=\frac{|W_{AB}^{XY}|}{a_Xa_Y}.\end{equation}

\begin{lemma}\label{mainlemma}
$(1)$~~For any $A_1,B_2\in\A$. In $\mathcal {D}\mathcal {H}(\mathcal{A})$ we have
$$[C_{A_1}][C_{B_2}^\ast]=\sum_{[X],[Y],[L]}v^{a}a_Lg_{LX}^{A_1}g_{YL}^{B_2}K_{\hat{\Omega}_{A_1}-\hat{\Omega}_X}K_{\hat{P}_{B_2}-\hat{P}_Y}^\ast[C_X\oplus C_Y^\ast],$$
where $$a=\lr{P_{A_1},\Omega_{B_2}}-\lr{\Omega_{A_1},P_{B_2}}+\lr{\hat{P}_{B_2}+\hat{\Omega}_X-\hat{\Omega}_{A_1}-\hat{P}_Y,\hat{X}-\hat{Y}};$$

$(2)$~~For any $A_2,B_1\in\A$. In $\mathcal {D}\mathcal {H}(\mathcal{A})$ we have
$$[C_{B_1}^\ast][C_{A_2}]=\sum_{[X],[Y],[L']}v^{a'}a_{L'}g_{L'Y}^{B_1}g_{XL'}^{A_2}K_{\hat{P}_{{A}_2}-\hat{P}_X}K_{\hat{\Omega}_{B_1}-\hat{\Omega}_Y}^\ast[C_X\oplus C_Y^\ast],$$
where $$a'=\lr{P_{B_1},\Omega_{A_2}}-\lr{\Omega_{B_1},P_{A_2}}+\lr{\hat{P}_X+\hat{\Omega}_{B_1}-\hat{P}_{A_2}-\hat{\Omega}_{Y},\hat{X}-\hat{Y}}.$$
\end{lemma}
\begin{proof}
We only prove $(1)$, since $(2)$ is similar.
\begin{equation}\label{Iso}
\begin{split}
\Ext^1_{\mathcal{C}_2(\mathcal{A})}(C_{A_1},C_{B_2}^\ast)&\cong\Hom_{{\mathcal {K}}_2(\mathscr{P})}(C_{A_1},C_{B_2})\\
&\cong\Hom_{\mathcal{R}_2(\mathcal{A})}(A_1,B_2)\\
&\cong\oplus_{i\in\mathbb{Z}}\Hom_{D^b(\A)}(A_1,B_2[2i])\\
&\cong\Hom_{\A}(A_1,B_2),~~\mbox{since}~~\A~~\mbox{is}~~\mbox{hereditary.}
\end{split}\end{equation}

Consider an extension of $C_{A_1}$ by $C_{B_2}^\ast$
$$\eta: 0\longrightarrow C_{B_2}^\ast\longrightarrow L_{\bullet}\longrightarrow C_{A_1}\longrightarrow 0.$$ It induces a long exact sequence in homology
$$H_0(C_{B_2}^\ast)\longrightarrow H_0(L_\bullet)\longrightarrow H_0(C_{A_1})\longrightarrow
H_1(C_{B_2}^\ast)\longrightarrow H_1(L_\bullet)\longrightarrow H_1(C_{A_1}).$$
Writing $L_\bullet=C_X\oplus C_Y^\ast \oplus K_T\oplus K_W^\ast$ for some objects $X,Y\in\A$ and $T,W\in\P$, we obtain the following exact sequence
$$0\longrightarrow X\longrightarrow A_1\stackrel{\delta}{\lra} B_2\longrightarrow Y\longrightarrow 0,$$
where $\delta$ is determined by the equivalence class of $\eta$ via the canonical isomorphism in $(\ref{Iso})$
\begin{equation}\label{biaozhun}\Ext^1_{\mathcal{C}_2(\mathcal{A})}(C_{A_1},C_{B_2}^\ast)\cong\Hom_{\A}(A_1,B_2).\end{equation}
By considering the kernels of differentials in $C_{B_2}^\ast, L_\bullet$ and $C_{A_1}$, we get that
$$P_Y\oplus W\cong P_{B_2},~~\mbox{and~~then}~~~\Omega_X\oplus T\cong \Omega_{A_1}.$$
That is, $T$ and $W$ are uniquely determined by $X$ and $Y$ up to isomorphism, respectively.
The canonical isomorphism $(\ref{biaozhun})$ induces an isomorphism
\begin{equation*}
\Ext^1_{\mathcal{C}_2(\mathcal{A})}(C_{A_1},C_{B_2}^\ast)_{C_X\oplus C_Y^\ast\oplus K_T\oplus K_W^\ast}
\cong {}_X\Hom_{\A}(A_1,B_2)_Y.
\end{equation*}
Hence, \begin{equation*}\begin{split}|\Ext^1_{\mathcal{C}_2(\mathcal{A})}(C_{A_1},C_{B_2}^\ast)_{C_X\oplus C_Y^\ast\oplus K_T\oplus K_W^\ast}|=\frac{|W_{A_1B_2}^{XY}|}{a_Xa_Y}
=\sum_{[L]}a_Lg_{LX}^{A_1}g_{YL}^{B_2}.\end{split}\end{equation*}

Thus, \begin{equation*}\begin{split}
[C_{A_1}][C_{B_2}^\ast]=\sum_{[X],[Y],[L]}v^{\lr{P_{A_1},\Omega_{B_2}}-\lr{\Omega_{A_1},P_{B_2}}}a_Lg_{LX}^{A_1}g_{YL}^{B_2}
[C_X\oplus C_Y^\ast\oplus K_T\oplus K_W^\ast],
\end{split}\end{equation*}
here we have used that $$|\Hom_{\mathcal{C}_2(\mathcal{A})}(C_{A_1},C_{B_2}^\ast)|=q^{\lr{\Omega_{A_1},P_{B_2}}}.$$
By Lemma \ref{zhongjian},
\begin{equation*}\begin{split}
[C_{A_1}][C_{B_2}^\ast]&=\sum_{[X],[Y],[L]}v^{a}a_Lg_{LX}^{A_1}g_{YL}^{B_2}K_TK_W^\ast[C_X\oplus C_Y^\ast]\\
&=\sum_{[X],[Y],[L]}v^{a}a_Lg_{LX}^{A_1}g_{YL}^{B_2}K_{\hat{\Omega}_{A_1}-\hat{\Omega}_X}K_{\hat{P}_{B_2}-\hat{P}_Y}^\ast[C_X\oplus C_Y^\ast],
\end{split}\end{equation*}
where \begin{equation*}\begin{split}
a&=\lr{P_{A_1},\Omega_{B_2}}-\lr{\Omega_{A_1},P_{B_2}}+\lr{\hat{W}-\hat{T},\hat{X}-\hat{Y}}\\
&=\lr{P_{A_1},\Omega_{B_2}}-\lr{\Omega_{A_1},P_{B_2}}+\lr{\hat{P}_{B_2}+\hat{\Omega}_X-\hat{\Omega}_{A_1}-\hat{P}_Y,\hat{X}-\hat{Y}}.
\end{split}\end{equation*}
\end{proof}

\noindent$\mathbf{Proof~~of~~Main~~Theorem}$~~~~

\begin{equation*}\begin{split}
&\mbox{LHS~~of}~~(\ref{Drinfeld2})\\&=\sum_{\begin{smallmatrix}[A_1],[A_2],[B_2],\\ [L],[X],[Y]\end{smallmatrix}}v^{a+x}
g_{A_1A_2}^Ag_{A_2B_2}^Bg_{LX}^{A_1}g_{YL}^{B_2}
a_{A_2}a_LK_{\hat{A}_2-\hat{\Omega}_{A_1}}K_{-\hat{\Omega}_{B_2}}^\ast K_{\hat{\Omega}_{A_1}-\hat{\Omega}_X}K_{\hat{P}_{B_2}-\hat{P}_Y}^\ast[C_X\oplus C_Y^\ast]\\
&=\sum_{\begin{smallmatrix}[A_1],[A_2],[B_2],\\ [L],[X],[Y]\end{smallmatrix}}v^{\lr{\hat{B}_2+\hat{\Omega}_X-\hat{P}_Y-\hat{A}_2,\hat{X}-\hat{Y}}}
g_{A_1A_2}^Ag_{A_2B_2}^Bg_{LX}^{A_1}g_{YL}^{B_2}
a_{A_2}a_LK_{\hat{A}_2-\hat{\Omega}_X}K_{\hat{B}_2-\hat{P}_Y}^\ast[C_X\oplus C_Y^\ast]\\
&=\sum_{\begin{smallmatrix}[A_1],[A_2],[B_2],\\ [L],[X],[Y]\end{smallmatrix}}v^{\lr{\hat{Y}+\hat{L}+\hat{\Omega}_X-\hat{P}_Y-\hat{A}_2,\hat{X}-\hat{Y}}}
g_{A_1A_2}^Ag_{A_2B_2}^Bg_{LX}^{A_1}g_{YL}^{B_2}
a_{A_2}a_LK_{\hat{A}_2-\hat{\Omega}_X}K_{\hat{Y}+\hat{L}-\hat{P}_Y}^\ast[C_X\oplus C_Y^\ast]\\
&=\sum_{\begin{smallmatrix}[A_2],[L],\\ [X],[Y]\end{smallmatrix}}v^{\lr{\hat{L}+\hat{\Omega}_X-\hat{\Omega}_Y-\hat{A}_2,\hat{X}-\hat{Y}}}
g_{LXA_2}^Ag_{A_2YL}^B
a_{A_2}a_LK_{\hat{A}_2-\hat{\Omega}_X}K_{\hat{L}-\hat{\Omega}_Y}^\ast[C_X\oplus C_Y^\ast],
\end{split}\end{equation*}here we get the last equality by using the associativity formula in $(\ref{jiehe})$.

\begin{equation*}\begin{split}
&\mbox{LHS~~of}~~(\ref{Drinfeld2})\\&=\sum_{\begin{smallmatrix}[A'_1],[A'_2],[B'_1],\\ [L'],[X],[Y]\end{smallmatrix}}v^{a'+x'}
g_{A'_1A'_2}^Ag_{B'_1A'_1}^Bg_{L'Y}^{B'_1}g_{XL'}^{A'_2}
a_{A'_1}a_{L'}K_{\hat{A}'_1-\hat{\Omega}_{B'_1}}^\ast K_{-\hat{\Omega}_{A'_2}} K_{\hat{P}_{A'_2}-\hat{P}_X}K_{\hat{\Omega}_{B'_1}-\hat{\Omega}_Y}^\ast[C_X\oplus C_Y^\ast]\\
&=\sum_{\begin{smallmatrix}[A'_1],[A'_2],[B'_1],\\ [L'],[X],[Y]\end{smallmatrix}}v^{\lr{\hat{A}'_1+\hat{P}_X-\hat{\Omega}_Y-\hat{A}'_2,\hat{X}-\hat{Y}}}
g_{A'_1A'_2}^Ag_{B'_1A'_1}^Bg_{L'Y}^{B'_1}g_{XL'}^{A'_2}
a_{A'_1}a_{L'}K_{\hat{A}'_2-\hat{P}_X}K_{\hat{A}'_1-\hat{\Omega}_Y}^\ast[C_X\oplus C_Y^\ast]\\
&=\sum_{\begin{smallmatrix}[A'_1],[A'_2],[B'_1],\\ [L'],[X],[Y]\end{smallmatrix}}v^{\lr{\hat{A}'_1+\hat{P}_X-\hat{\Omega}_Y-\hat{X}-\hat{L}',\hat{X}-\hat{Y}}}
g_{A'_1A'_2}^Ag_{B'_1A'_1}^Bg_{L'Y}^{B'_1}g_{XL'}^{A'_2}
a_{A'_1}a_{L'}K_{\hat{X}+\hat{L}'-\hat{P}_X}K_{\hat{A}'_1-\hat{\Omega}_Y}^\ast[C_X\oplus C_Y^\ast]\\
&=\sum_{\begin{smallmatrix}[L'], [A'_1],\\ [X],[Y]\end{smallmatrix}}v^{\lr{\hat{A}'_1+\hat{\Omega}_X-\hat{\Omega}_Y-\hat{L}',\hat{X}-\hat{Y}}}
g_{A'_1XL'}^Ag_{L'YA'_1}^B
a_{L'}a_{A'_1}K_{\hat{L}'-\hat{\Omega}_X}K_{\hat{A}'_1-\hat{\Omega}_Y}^\ast[C_X\oplus C_Y^\ast].
\end{split}\end{equation*}
Identifying $[L]$ and $[A_2]$ in LHS of $(\ref{Drinfeld2})$ with $[A'_1]$ and $[L']$ in RHS of $(\ref{Drinfeld2})$, respectively, we obtain that
\begin{equation*}
\mbox{LHS~~of}~~(\ref{Drinfeld2})=\mbox{RHS~~of}~~(\ref{Drinfeld2}).\end{equation*}
\fin

\section{Appendix: A remark on modified Ringel--Hall algebras}
In this section, we briefly give a proof of Main Theorem $4.11$ in \cite{LP} by using the  associativity formula of Hall algebras without \cite[Lemma 4.10]{LP}. Let $\A$ be a finitary hereditary abelian $k$-category, and we do not assume that it has enough projectives. Inspired by the works of Bridgeland \cite{Br} and Gorsky \cite{Gor2}, Lu and Peng \cite{LP} introduced an algebra $\mathcal {M}\mathcal {H}_{\mathbb{Z}/2,tw}(\A)$, called the {\em modified Ringel--Hall algebra} of $\A$, with the purpose of generalizing Bridgeland's construction to any hereditary abelian categories satisfying certain finiteness conditions.
For unexplained notations (such as $C_X,C_Y^\ast,K_X$, and $K_Y^\ast$) concerning the modified Ringel--Hall algebra we refer to \cite{LP}.

For any $A,B\in\A$, we consider the action of the group $\Aut_{\A}(A)\times\Aut_{\A}(B)$ on $\Hom_{\A}(A,B)$, which is defined by the following commutative diagram
$$\xymatrix{A\ar[r]^-g\ar[d]_-s&B\ar[d]^-t\\
A\ar[r]^-{g'}&B,}$$ for any $g\in\Hom_{\A}(A,B), s\in\Aut_{\A}(A)$ and $t\in\Aut_{\A}(B)$. We denote by $O_g$ the orbit of $g$ for any $g\in\Hom_{\A}(A,B)$, and set $O=\{O_g~|~g\in\Hom_{\A}(A,B)\}$. Clearly, we have the following equations
\begin{equation}\label{huafeng}\sum_{\begin{smallmatrix}O_g\in O: g\in{\rm Hom}_{\A}(A,B)\\
{\rm ker} g\cong X\\
{\rm coker} g\cong Y\end{smallmatrix}}|O_g|=|{}_X\Hom_{\A}(A,B)_Y|=\sum_{[L]}a_Lg_{LX}^{A}g_{YL}^{B}.\end{equation}
By $(\ref{huafeng})$, we reformulate \cite[Lemma 4.9]{LP} as follows.
\begin{lemma}{\rm(\cite{LP})}\label{AL1}
$(1)$~~For any $A_1,B_2\in\A$. In $\mathcal {M}\mathcal {H}_{\mathbb{Z}/2,tw}(\A)$ we have
$$[C_{A_1}]\ast[C_{B_2}^\ast]=\sum_{[X],[Y],[L]}v^{\lr{\hat{Y}-\hat{X},\hat{B}_2-\hat{Y}}}a_Lg_{LX}^{A_1}g_{YL}^{B_2}[C_X\oplus C_Y^\ast]\ast[K_{\hat{B}_2-\hat{Y}}^\ast];$$

$(2)$~~For any $A_2,B_1\in\A$. In $\mathcal {M}\mathcal {H}_{\mathbb{Z}/2,tw}(\A)$ we have
$$[C_{B_1}^\ast]\ast[C_{A_2}]=\sum_{[X],[Y],[L']}v^{\lr{\hat{X}-\hat{Y},\hat{A}_2-\hat{X}}}a_{L'}g_{L'Y}^{B_1}g_{XL'}^{A_2}[C_X\oplus C_Y^\ast]\ast[K_{\hat{A}_2-\hat{X}}].$$
\end{lemma}
\begin{remark}
In Lemma $\ref{AL1}$, we have employed the additive Euler form rather than the multiplicative one used in \cite{LP}.
\end{remark}

Using Lemma $\ref{AL1}$, we simplify the proof of \cite[Thm. 4.11]{LP} as follows.
\begin{equation*}\begin{split}
&\mbox{LHS~~of~~the~~equation}~~(19)~~\mbox{in~~\cite{LP}}\\&=
\sum_{[A_1],[A_2],[B_2]}v^{\lr{A_1,A_2}+\lr{A_2,B_2}-(A_2,B_2)}g_{A_1A_2}^Ag_{A_2B_2}^Ba_{A_2}[C_{A_1}]\ast[C_{B_2}^\ast]\ast K_{\hat{A}_2}\\&=
\sum_{\substack{[A_1],[A_2],[B_2],\\
[X],[Y],[L]}}v^{\lr{\hat{A}_1-\hat{B}_2,\hat{A}_2}+\lr{\hat{Y}-\hat{X},\hat{B}_2-\hat{Y}}}
g_{A_1A_2}^Ag_{A_2B_2}^Bg_{LX}^{A_1}g_{YL}^{B_2}a_{A_2}a_{L}[C_X\oplus C_Y^\ast]\ast K_{\hat{B}_2-\hat{Y}}^\ast\ast K_{\hat{A}_2}\\&
\xlongequal[\hat{B}_2=\hat{Y}+\hat{L}]{\hat{A}_1=\hat{L}+\hat{X}}\sum_{\substack{[A_2],[L],\\
[X],[Y]}}v^{\lr{\hat{X}-\hat{Y},\hat{A}_2-\hat{L}}}g_{LXA_2}^Ag_{A_2YL}^Ba_{A_2}a_L[C_X\oplus C_Y^\ast]\ast K_{\hat{L}}^\ast \ast K_{\hat{A}_2},
\end{split}\end{equation*}here we have used the associativity formula in $(\ref{jiehe})$.

\begin{equation*}\begin{split}
&\mbox{RHS~~of~~the~~equation}~~(19)~~\mbox{in~~\cite{LP}}\\&=
\sum_{[\tilde{A}_1],[\tilde{A}_2],[\tilde{B}_1]}v^{\lr{\tilde{A}_1,\tilde{A}_2}+\lr{\tilde{B}_1,\tilde{A}_1}-(\tilde{A}_1,\tilde{A}_2)}g_{\tilde{A}_1\tilde{A}_2}^Ag_{\tilde{B}_1\tilde{A}_1}^Ba_{\tilde{A}_1}[C_{\tilde{B}_1}^\ast]\ast[C_{\tilde{A}_2}]\ast K_{\hat{\tilde{A}}_1}^\ast\\&=
\sum_{\substack{[\tilde{A}_1],[\tilde{A}_2],[\tilde{B}_1],\\
[X],[Y],[L']}}v^{\lr{\hat{\tilde{B}}_1-\hat{\tilde{A}}_2,\hat{\tilde{A}}_1}+\lr{\hat{X}-\hat{Y},\hat{\tilde{A}}_2-\hat{X}}}
g_{\tilde{A}_1\tilde{A}_2}^Ag_{\tilde{B}_1\tilde{A}_1}^Bg_{L'Y}^{\tilde{B}_1}g_{XL'}^{\tilde{A}_2}a_{L'}a_{\tilde{A}_1}[C_X\oplus C_Y^\ast]\ast K_{\hat{\tilde{A}}_2-\hat{X}}\ast K_{\hat{\tilde{A}}_1}^\ast\\&
\xlongequal[\hat{\tilde{A}}_2=\hat{X}+\hat{L'}]{\hat{\tilde{B}}_1=\hat{L'}+\hat{Y}}\sum_{\substack{[\tilde{A}_1],[L'],\\
[X],[Y]}}v^{\lr{\hat{X}-\hat{Y},\hat{L'}-\hat{\tilde{A}}_1}}g_{\tilde{A}_1XL'}^Ag_{L'Y\tilde{A}_1}^Ba_{L'}a_{\tilde{A}_1}[C_X\oplus C_Y^\ast]\ast K_{\hat{\tilde{A}}_1}^\ast \ast K_{\hat{L'}}.
\end{split}\end{equation*}
Identifying $[L]$ and $[A_2]$ in LHS with $[\tilde{A}_1]$ and $[L']$ in RHS, respectively, we obtain that
\begin{equation*}
\mbox{LHS~~of~~the~~equation}~~(19)~~\mbox{in~~\cite{LP}}=\mbox{RHS~~of~~the~~equation}~~(19)~~\mbox{in~~\cite{LP}}.\end{equation*}

\begin{remark}
The preliminary part (Lemma \ref{mainlemma}) for the proof of Main Theorem is similar to \cite[Lemma 4.9]{LP}, but the calculation methods are not the same. In this note, we explicitly work out the coefficients in the summation via Hall numbers. While, Lu and Peng introduced orbit sets $O_g$ and express the coefficients by $|O_g|$. For the conclusive part, we are reduced to use the associativity formula of Hall algebras, and Lu and Peng introduced once more two sets and give a characterization of the cardinalities of these sets (\cite[Lemma 4.10]{LP}), then they are reduced to the equality of the cardinalities of these two sets. In some sense, we avoid computing the cardinalities of the sets defined by Lu and Peng, this work seems to be equivalent to the proof of the associativity of Hall algebras. However, the advantage of their proof is that we have better understanding of the essence of the coefficients in the commutator relation $(\ref{Drinfeld2})$.
\end{remark}

%
%
%
%
%

\section*{Acknowledgments}

The author is grateful to Ming Lu for his careful reading and helpful comments, and Professor Bangming Deng for his patience guidance and valuable comments. He also would like to thank Shiquan Ruan and Panyue Zhou for their help in his study and life.

\end{document}